\documentclass[12pt]{amsart}

\usepackage{amsmath,amssymb,amsthm}
\usepackage[margin=1.15in]{geometry}
\usepackage{hyperref}
\usepackage{xcolor}
\newtheorem{theorem}{Theorem}
\newtheorem{proposition}[theorem]{Proposition}
\newtheorem{lemma}[theorem]{Lemma}
\newtheorem*{conjecture}{Conjecture}

\newcommand{\rad}{\operatorname{rad}}

\theoremstyle{remark}
\newtheorem{remark}[theorem]{Remark}

\numberwithin{equation}{section}

\title[ABC implies that Ramanujan's tau function misses almost all primes]{ABC implies that Ramanujan's tau function misses almost all primes}

\newif\ifmanyauthors
\manyauthorstrue 
\ifmanyauthors
  \usepackage[foot]{amsaddr}
  \usepackage{textcomp}
  \newcommand{\dmd}{\ensuremath{\diamond}}
  \author{David Kurniadi Angdinata\textsuperscript{\dag}}
  \email{davidang@axiommath.ai}
  \author{Evan Chen\textsuperscript{\dag}}
  \email{evan@axiommath.ai}
  \author{Chris Cummins\textsuperscript{*}}
  \email{chris@axiommath.ai}
  \author{Ben Eltschig\textsuperscript{*}}
  \email{ben@axiommath.ai}
  \author{Dejan Grubisic\textsuperscript{*}}
  \email{dejan@axiommath.ai}
  \author{Leopold Haller\textsuperscript{*}}
  \email{leo@axiommath.ai}
  \author{Letong Hong\textsuperscript{\dmd}}
  \email{carina@axiommath.ai}
  \author{Andranik Kurghinyan\textsuperscript{*}}
  \email{andranik@axiommath.ai}
  \author{Kenny Lau\textsuperscript{*}}
  \email{kenny@axiommath.ai}
  \author{Hugh Leather\textsuperscript{*}}
  \email{hughleat@gmail.com}
  \author{Seewoo Lee\textsuperscript{\dag}}
  \email{seewoo@axiommath.ai}
  \author{Simon Mahns\textsuperscript{*}}
  \email{simon@axiommath.ai}
  \author{Aram H. Markosyan\textsuperscript{*}}
  \email{am@axiommath.ai}
  \author{Rithikesh Muddana\textsuperscript{*}}
  \email{rithikesh@axiommath.ai}
  \author{Ken Ono\textsuperscript{\dag}}
  \email{ken@axiommath.ai}
  \author{Manooshree Patel\textsuperscript{*}}
  \email{manooshree@axiommath.ai}
  \author{Gaurang Pendharkar\textsuperscript{*}}
  \email{gaurang@axiommath.ai}
  \author{Vedant Rathi\textsuperscript{*}}
  \email{vedant@axiommath.ai}
  \author{Alex Schneidman\textsuperscript{*}}
  \email{alex@axiommath.ai}
  \author{Volker Seeker\textsuperscript{*}}
  \email{volker@axiommath.ai}
  \author{Shubho Sengupta\textsuperscript{\dmd}}
  \email{shubho@axiommath.ai}
  \author{Ishan Sinha\textsuperscript{*}}
  \email{ishan@axiommath.ai}
  \author{Jimmy Xin\textsuperscript{*}}
  \email{jimmy@axiommath.ai}
  \author{Jujian Zhang\textsuperscript{\dag*}}
  \email{jujian@axiommath.ai}
\else
  \author{David Kurniadi Angdinata}
  \email{davidang@axiommath.ai}
  \author{Evan Chen}
  \email{evan@axiommath.ai}
  \author{Seewoo Lee}
  \email{seewoo@axiommath.ai}
  \author{Ken Ono}
  \email{ken@axiommath.ai}
  \author{Jujian Zhang}
  \email{jujian@axiommath.ai}
\fi

\dedicatory{In celebration of Krishnaswami Alladi's 70th birthday}

\date{\today}

\begin{document}

\maketitle

\ifmanyauthors
\begin{center}
  \footnotesize
  Authors are listed alphabetically. \\
  \textsuperscript{\dag}Mathematical contributor,
  \textsuperscript{*}Engineering contributor,
  \textsuperscript{\dmd}Principal investigator.
\end{center}
\fi

\begin{abstract}
Lehmer conjectured that Ramanujan's tau-function never vanishes. In a related direction,
a folklore conjecture asserts that infinitely many primes arise as absolute values of Ramanujan's
tau-function. Recently, Xiong showed that these prime values form a subset of the primes with density
at most $2/11$. Assuming the $abc$ Conjecture,
we prove the stronger upper bound
\[ S(X) := \#\{\ell\le X:\ \ell\ \text{prime and } |\tau(n)|=\ell \text{ for some } n\ge 1\}
  = O(X^{13/22}), \]
which implies that Ramanujan's tau-function misses a density 1 subset of the primes. We give
a heuristic suggesting that $S(X)$ should nevertheless be infinite, with predicted order of magnitude
\[
S(X)\asymp \frac{X^{\frac{1}{11}}}{(\log X)^2}.
\]
The main engine in this note was formalized and produced automatically in Lean/Mathlib by
AxiomProver from a natural-language statement of the problem.
\end{abstract}

\section{Introduction and statement of results}

Ramanujan's discriminant modular form
\begin{equation}\label{eq:delta}
\Delta(z)=q\prod_{n\ge 1}(1-q^n)^{24}=\sum_{n\ge 1}\tau(n)q^n,\qquad (q=e^{2\pi i z}),
\end{equation}
is the normalized weight $12$ cusp form on $\mathrm{SL}_2(\mathbb Z)$ with integer Fourier coefficients
$\tau(n)$ \cite{Ramanujan1916}. Lehmer famously conjectured that $\tau(n)\neq 0$ for all $n\ge 1$
\cite{Lehmer1947}, and he also initiated the study of primality in the image of $\tau$ \cite{Lehmer1965}.
Serre proved that the primes $p$ with $\tau(p)=0$ have natural density $0$ \cite{Serre1981}.
This result was further refined by Thorner and Zaman \cite{ThornerZaman}.

Beyond nonvanishing, one may ask which integers lie in the image of $\tau$.
For fixed odd integers $\alpha$, Murty--Murty--Shorey proved that $\tau(n)=\alpha$ has only finitely many
solutions \cite{MurtyMurtyShorey1987}. In recent years, there has been substantial interest in finding
\emph{explicit} omitted values of $\tau$. A foundational structural input for many such results is a
prime-power criterion of Balakrishnan--Craig--Tsai and one of the present authors \cite{BCOT2023}.
Using this criterion, they \cite{BCOT2023} proved, for $n\ge 2$, that
\[
\tau(n)\notin\{\pm 1,\pm 3,\pm 5,\pm 7,\pm 13,\pm 17,-19,\pm 23,\pm 37,\pm 691\}.
\]
After that, many authors have used this criterion to find additional omitted values (for example, see \cite{AmirHong2022, DembnerJain2021, HanadaMadhukara2021}).
Furthermore, Bennett--Gherga--Patel--Siksek \cite{BGPSS2022}
proved that $\tau(n)\neq \pm \ell^m$ for all odd primes $3\le \ell<100$ and all integers $m\ge 1$.
Regarding even values, Balakrishnan--Ono--Tsai \cite{BalakrishnanOnoTsai2022} produced the first explicit even
integers known not to occur as $\tau$-values.

A folklore conjecture (for example, see \cite{BGPSS2022, LygerosRozier2013}) asserts that the $\tau$-values include infinitely many primes up to sign.
In the complementary direction, Xiong \cite{Xiong2025} bounded (from above) the density of such values.
Namely, he investigated the function
\begin{equation}\label{Count}
S(X):=\#\{\ell\le X:\ \ell \text{ prime with } |\tau(n)|=\ell\ \text{ for some } n\}.
\end{equation}
He proved that in $18$ of the non-zero residue classes\footnote{The obstruction pertains
to the primes in the four congruence classes
$\ell\equiv 1,3,5,22\pmod{23}.$}
 modulo $23$, prime $\tau$-values form a thin set with
$$
\lim_{X\rightarrow +\infty} \frac{S(X)}{\pi(X)} \leq \frac{2}{11},
$$
where $\pi(X)$ denotes the number of primes $\leq X.$

In this note, we address the primes in these classes conditionally.
We let
\begin{equation}
\mathcal P:=\{\pm \ell:\ \ell \text{ is an odd prime}\}
\qquad  {\text {and}} \qquad
\mathcal V:=\{\tau(n):n\ge 1\}.
\end{equation}
To study $\mathcal P\cap \mathcal V$, the odd primes (up to sign) that occur as $\tau$-values,
we employ the celebrated $abc$ Conjecture of Masser and Oesterl\'e \cite[Exp.~694]{Oesterle1988}.

\begin{conjecture}[$abc$ Conjecture]\label{conj:abc}
For every $\varepsilon>0,$ there exists a constant $C_\varepsilon>0$ such that
for all coprime nonzero integers
$a, b, c$ with $a+b=c$ one has
\[
|c|\le C_\varepsilon\, \rad(abc)^{1+\varepsilon},
\]
where $\rad(abc):=\prod_{p\mid abc} p$ is the product of distinct primes dividing $abc.$
\end{conjecture}

Assuming this conjecture, we obtain the following improvement to Xiong's theorem.

\begin{theorem}\label{thm:main}
Assuming the $abc$ Conjecture, we have
$$
S(X) = O(X^{\frac{13}{22}}).
$$
\end{theorem}

\begin{remark} Theorem~\ref{thm:main} implies that
 Ramanujan's tau-function misses almost all primes under the $abc$ Conjecture.
 In other words, we have
\[
\frac{S(X)}{\pi(X)}\longrightarrow 0,
\]
which follows by the Prime Number Theorem
$$
\lim_{X\rightarrow +\infty}\frac{\pi(X)}{X/\log X}=1.
$$
For completeness, we point out that there are prime values of Ramanujan's tau-function.
Indeed, Lehmer \cite{Lehmer1965} found that
$$
\tau(251^2)= -80561663527802406257321747.
$$
\end{remark}

\begin{remark}
Theorem~\ref{thm:main} holds in much greater generality.
The same conclusion holds for the coefficients of any normalized Hecke eigenform
on $\mathrm{SL}_2(\mathbb{Z})$ with integer coefficients.
To see that, one notes that the same Hecke relations hold, which means that the proof of Theorem~1.2 of \cite{Xiong2025} applies {\it mutatis mutandis}.
One then uses the fact that every integer Hecke eigenvalue $\lambda(p)$, for odd primes $p$ in level 1, is even (for example, see \cite{Hatada}). This implies the direct analog of Proposition~8 in this note.
\end{remark}

Assuming the $abc$ Conjecture, Theorem~\ref{thm:main} establishes that the $\tau$-values can only include a density zero subset of the primes (up to sign). In Section~\ref{heuristic}, we address the conjecture that there are still infinitely many such values.
We employ the truth of the Sato-Tate Conjecture to suggest that
\[
S(X)\asymp \frac{X^{\frac{1}{11}}}{(\log X)^2}.
\]

This paper is organized as follows. In Section~\ref{sec:engine}, we prove the main engine of the paper, which are two analytic estimates for the number of integer points that lie near two specific hyperelliptic curves under the $abc$ Conjecture. In Section~\ref{sec:proof},  we then use these estimates to prove Theorem~\ref{thm:main}. As mentioned above, in Section~\ref{heuristic} we suggest an analytic estimate for $S(X)$. Finally, in Section~\ref{sec:AI}, we describe the protocol we employed to formalize and automatically generate the two estimates from a natural-language statement of the problem.

\section*{Acknowledgements}
\noindent The authors thank Tim Browning, Will Craig and Wei-Lun Tsai
for conversations related to this note.
We also thank Bouyan Xiong for his comments on the first draft of this paper,
particularly for pointing out that in Proposition~\ref{prop:xiong} we can
use the exponent $\frac{1}{2}$ instead of $\frac{9}{10}$ in the inequality for sufficiently large $N.$
This led directly to an improvement in Theorem~\ref{thm:main},
where we achieve $S(X) \ll X^{13/22}$;
the first version of this paper had the weaker result $S(X) \ll X^{9/10} \log X$.
Finally, the authors thank the anonymous referees for helpful comments.

\ifmanyauthors
\else
This paper describes a test case for AxiomProver,
an autonomous system that is currently under development.
The project engineering team is
Chris Cummins,
Ben Eltschig,
GSM,
Dejan Grubisic,
Leopold Haller,
Letong Hong (principal investigator),
Andranik Kurghinyan,
Kenny Lau,
Hugh Leather,
Simon Mahns,
Aram H. Markosyan,
Rithikesh Muddana,
Manooshree Patel,
Gaurang Pendharkar,
Vedant Rathi,
Alex Schneidman,
Volker Seeker,
Shubho Sengupta (principal investigator),
Ishan Sinha,
Jimmy Xin,
and Jujian Zhang.
\fi

\section{The main engine}\label{sec:engine}
The overall strategy of the proof of Theorem~\ref{thm:main} can be summarized as follows.
\begin{itemize}
  \item As we will see in Proposition~\ref{prop:X2k},
  $|\tau(n)|$ can only be prime when $n = p^{2k}$ for some prime $p$ and $k \ge 1$.
  \item We split off the contributions into the cases $k \ge 3$ and $k = 1, 2$.
  \item For $k \ge 3$, our bound comes from Xiong's result (which is Proposition~\ref{prop:xiong}).
  \item For $k = 1, 2$ we translate solutions to $\tau(p^{2k}) = \pm \ell$ to integral points
  on a hyperelliptic curve that we proceed to count.
  We comment that the main term will arise from the case $k = 1$ when $p < X^{2/11}$.
\end{itemize}

The following lemma makes precise the strategy we described above,
and is a straightforward consequence
of Xiong's Theorem (see the proof of Theorem~1.2 of \cite{Xiong2025}).
We defer the proof of Lemma~\ref{lem:reduction} to the next section.

\begin{lemma}\label{lem:reduction}
For the functions
\begin{displaymath}
\begin{split}
E_2(X)&:=\#\{(x,y)\in\mathbb Z_{\ge 1}\times\mathbb Z:\ x> X^{2/11},\ 1\le |x^{11}-y^2|\le X\},\\
E_4(X)&:=\#\{(x,u)\in\mathbb Z_{\ge 1}\times\mathbb Z:\ x> X^{1/11},\ 1\le |5x^{22}-u^2|\le 4X\},
\end{split}
\end{displaymath}
as $X\rightarrow +\infty$ we have
\[
S(X)\ \ll\ X^{13/22}\ +\ E_2(X)\ +\ E_4(X).
\]
\end{lemma}

Theorem~\ref{thm:main} then follows from the next result.

\begin{lemma}\label{lem:main_engine}
Assume the $abc$ Conjecture.
\begin{enumerate}
\item[(i)] For every $\eta>0$, we have that
$$E_2(X)\ll_{\eta} X^{4/9+\eta}.
$$
\item[(ii)] For every $\eta>0$, we have that
$$E_4(X)\ll_{\eta} X^{1/5+\eta}.
$$
\end{enumerate}
\end{lemma}

We defer the proof of Lemma~\ref{lem:main_engine} to the next section.
For the rest of Section~\ref{sec:engine}
we describe the inputs that will be used to prove Lemma~\ref{lem:reduction}
and Lemma~\ref{lem:main_engine}.

\subsection{The prime-power reduction}\label{subsec:21}
We begin with the prime-power reduction which underlies the recent work on variants of Lehmer's Conjecture,
which is adapted in Xiong's setup \cite[\S1--2]{Xiong2025}.
We recall the standard Hecke relations (see Mordell \cite{Mordell1917})
and Deligne's bound (see \cite[Theorem 8.2]{Deligne1974} and \cite{Deligne1980}),
stated by Xiong as \cite[Thm.~2.1]{Xiong2025}:
\begin{equation}
  \label{eq:hecke}
  \begin{aligned}
    \tau(mn) &=\tau(m)\tau(n) &\text{if }(m,n)=1, \\
    \tau(p^{r+1}) &= \tau(p)\tau(p^r)-p^{11}\tau(p^{r-1}) &\text{if }r\ge 1,
  \end{aligned}
\end{equation}
and $|\tau(p)|\le 2p^{11/2}$ for primes $p$.
Xiong also records the following parity fact (equivalently, $\Delta$ has trivial
residual mod $2$ Galois representation) \cite[Prop.~2.2]{Xiong2025}.

\begin{proposition}\label{prop:parity}
  We have that $\tau(n)$ is odd if and only if $n$ is an odd square.
\end{proposition}
\begin{proof}
By direct calculation, we have
\[ \Delta(z):=q\prod_{n=1}^{\infty} (1-q^n)^{24}\equiv q\prod_{n=1}^{\infty} (1-q^{8n})^3\pmod 2. \]
The claim now follows immediately from the classical Jacobi identity
\[ \prod_{n=1}^{\infty} (1-q^n)^3 =\sum_{k=0}^{\infty} (-1)^k (2k+1)q^{\frac{k^2+k}{2}}. \qedhere \]
\end{proof}
As in \cite[\S2]{Xiong2025}, we define
\[ X_k:=\{\tau(p^k):\ p\ \text{prime}\}. \]
\begin{proposition}
 \label{prop:X2k}
  If $|\tau(n)|$ is an odd prime, then
  \[ |\tau(n)| \in \bigcup_{k \ge 1} X_{2k}. \]
\end{proposition}
\begin{proof}
By Hecke multiplicativity in \eqref{eq:hecke}, together with the fact that $\tau(n)\neq \pm 1$ for $n\ge 2$
(for example, \cite[Thm.~1.1]{BalakrishnanCraigOno2022}),
if $|\tau(n)|$ is an odd prime then $n$ must be a prime power.
This yields the result when combined with Proposition~\ref{prop:parity}.
\end{proof}
\begin{remark}
  Other variants of Proposition~\ref{prop:X2k} appear in the literature.
  For example, Lygeros--Rozier showed that if $\tau(n)$ is an odd prime,
  then $n=p^{q-1}$ with $p$ and $q$ odd primes \cite{LygerosRozier2013}.
  And Balakrishnan--Craig--Ono--Tsai showed that
  if $|\tau(n)| = \ell^m$ for an odd prime $\ell$ with $\ell \nmid \tau(\ell)$,
  then $n = p^{d-1}$ where $p$ is an odd prime and $d$ is an odd prime divisor of $\ell(\ell^2-1)$
  \cite[Thm.~1.1]{BCOT2023}.
\end{remark}

The key quantitative input from Xiong is the following bound for $k\ge 3$.
\begin{proposition}[Xiong]\label{prop:xiong}
For all sufficiently large $N$ and all integers
$k$ with $3\le k<\frac{\log N}{2\log 2}$, we have
\[
  \#\bigl(\mathcal P\cap X_{2k}\cap[-N,N]\bigr)\ \ll N^{1/2}.
\]
Moreover, for $k\ge \frac{\log N}{2\log 2}$ one has $X_{2k}\cap[-N,N]=\varnothing$.
\end{proposition}
\begin{proof}
  The first statement is \cite[Prop.~5.4]{Xiong2025}.
  The second statement appears in a comment at the end of \S4 of \emph{op.\ cit.}
\end{proof}

We will combine Proposition~\ref{prop:xiong} with elementary bounds for the remaining cases $k=1,2$.

\subsection{Completion via hyperelliptic twists: the small exponents $X_2$ and $X_4$}\label{subsec:22}

\subsubsection*{Case of $X_2$ and the twists $y^2=x^{11}\pm \ell$}

From \eqref{eq:hecke} we have, for primes $p$,
\begin{equation}\label{eq:p2}
\tau(p^2)=\tau(p)^2-p^{11}.
\end{equation}
If $\tau(p^2)=\pm \ell$ for an odd prime $\ell$, then
\begin{equation}\label{eq:Ccurve}
\tau(p)^2=p^{11}\pm \ell,
\end{equation}
so $(x,y)=(p,\tau(p))$ is an integer point on the hyperelliptic curve
\[
C_{\ell}^{\pm}:\quad y^2=x^{11}\pm \ell.
\]
Thus primes $\ell$ arising from $X_2$ are contained in the set of prime parameters for which $C_{\ell}^{\pm}$
has an integer point.

\begin{remark} This is a slight abuse of terminology, as we don't consider the point at infinity on hyperelliptic curves.
\end{remark}

\begin{lemma}\label{lem:X2count}
If we let
\[
A_2(X):=\#\{\ell\le X:\ \ell \text{ prime and } C_{\ell}^{+}(\mathbb Z)\neq \varnothing\ \text{or}\ C_{\ell}^{-}(\mathbb Z)\neq \varnothing\},
\]
then we have
\[
A_2(X)\ll X^{13/22}+E_2(X),
\]
where
\[
E_2(X):=\#\{(x,y)\in\mathbb Z_{\ge 1}\times\mathbb Z:\ x> X^{2/11},\ 1\le |x^{11}-y^2|\le X\}.
\]
\end{lemma}

\begin{proof}
If $C_{\ell}^{\pm}(\mathbb Z)\neq\varnothing$ with $\ell\le X$,
then there exist integers $x\ge 1$ and $y$ with
$1\le |x^{11}-y^2|\le X$, and $\ell=|x^{11}-y^2|$ is prime. Hence
\[
A_2(X)\ \le\ N_2(X):=\#\{(x,y)\in\mathbb Z_{\ge 1}\times \mathbb Z:\ 1\le |x^{11}-y^2|\le X\}.
\]
Fix $x\ge 1$. The condition $|x^{11}-y^2|\le X$ implies
\[
x^{11}-X\ \le\ y^2\ \le\ x^{11}+X.
\]

If $x\le (2X)^{1/11}$, then $x^{11}+X\le 3X$, so $|y|\le (3X)^{1/2}.$  Therefore, the number of possible $y$'s appearing in $N_2(X)$, given $x$, satisfies
$\#\{y\}\ll X^{1/2}$.
Summing over $x\le (2X)^{1/11}$ gives
\[
\sum_{x\le (2X)^{1/11}}\#\{y\}\ \ll\ X^{1/11}\cdot X^{1/2}=X^{13/22}.
\]

If $x> (2X)^{1/11}$, then $x^{11}-X\ge \tfrac12 x^{11}$ and the interval for $|y|$ has length
\[
\sqrt{x^{11}+X}-\sqrt{x^{11}-X}
=\frac{2X}{\sqrt{x^{11}+X}+\sqrt{x^{11}-X}}
\ll \frac{X}{x^{11/2}}.
\]
Thus, arguing as above, we have $\#\{y\}\ll X/x^{11/2}+1$. In the range $(2X)^{1/11}<x\le X^{2/11}$ one has $X/x^{11/2}\ge 1$, so
$\#\{y\}\ll X/x^{11/2}$ there. Hence, we find that
\[
\sum_{(2X)^{1/11}<x\le X^{2/11}}\#\{y\}\ \ll\ \sum_{(2X)^{1/11}<x\le X^{2/11}}\frac{X}{x^{11/2}}
\ll\ X\int_{X^{1/11}}^\infty t^{-11/2}\,dt
\ll\ X\cdot (X^{1/11})^{-9/2}
= X^{13/22}.
\]
For $x> X^{2/11}$, the $y$-interval has length $\ll X/x^{11/2}<1$,
so there are at most $O(1)$ possibilities for $y$ (usually corresponding to just $\pm y$).
The total contribution from this ``sub-unit length'' regime is precisely the error term $E_2(X)$.
Combining the preceding estimates gives
\[
N_2(X)\ll X^{13/22}+E_2(X),
\]
and since $A_2(X)\le N_2(X)$ this proves the lemma.
\end{proof}

\subsubsection*{Case of $X_4$ and the twists $y^2=5x^{22}\pm 4\ell$}

Iterating \eqref{eq:hecke} gives
\begin{equation}\label{eq:p4}
\tau(p^4)=\tau(p)^4-3p^{11}\tau(p)^2+p^{22}.
\end{equation}
Set $X=p^{11}$ and $Y=\tau(p)^2$. Then \eqref{eq:p4} becomes
\[
\tau(p^4)=Y^2-3XY+X^2.
\]
If $\tau(p^4)=\pm \ell$ is an odd prime, we obtain the quadratic equation
\begin{equation}\label{eq:thuequad}
Y^2-3XY+X^2=\pm \ell,\qquad (X=p^{11},\ Y=\tau(p)^2).
\end{equation}
Completing the square, we obtain
\[
(2Y-3X)^2 = 4(Y^2-3XY+X^2)+5X^2 = 5X^2 \pm 4\ell.
\]
Therefore, $(x,u)=(p,\,2\tau(p)^2-3p^{11})$ is an integer point on
\[
H_{\ell}^{\pm}:\quad u^2 = 5x^{22}\pm 4\ell.
\]
Again, primes $\ell$ arising from $X_4$ are contained in the set of prime
parameters for which $H_{\ell}^{\pm}$
has an integer point.

\begin{lemma}\label{lem:X4count}
If we let
\[
A_4(X):=\#\{\ell\le X:\ \ell \text{ prime and } H_{\ell}^{+}(\mathbb Z)\neq \varnothing\ \text{or}\ H_{\ell}^{-}(\mathbb Z)\neq \varnothing\},
\]
then we have
\[
A_4(X)\ll X^{6/11}+E_4(X),
\]
where
\[
E_4(X):=\#\{(x,u)\in\mathbb Z_{\ge 1}\times\mathbb Z:\ x> X^{1/11},\ 1\le |5x^{22} - u^2|\le 4X\}.
\]
\end{lemma}

\begin{proof}
If $H_{\ell}^{\pm}(\mathbb Z)\neq\varnothing$ with $\ell\le X$,
then there exist integers $x\ge 1$ and $u$ with
$u^2=5x^{22}\pm 4\ell$, hence $1\le |5x^{22} - u^2|\le 4X$ and $\ell=|5x^{22} - u^2|/4$ is prime. Thus
\[
A_4(X)\ \le\ N_4(X):=\#\{(x,u)\in\mathbb Z_{\ge 1}\times \mathbb Z:\ 1\le |5x^{22} - u^2|\le 4X\}.
\]

Fix $x\ge 1$.
If $x\le X^{1/22}$, then $5x^{22}\le 5X$, so $u^2\le 5x^{22}+4X\le 9X$ and $|u|\le 3X^{1/2}$.
Hence, arguing as in the proof of Lemma~\ref{lem:X2count}, we have $\#\{u\}\ll X^{1/2}$, and summing over $x\le X^{1/22}$ gives
\[
\sum_{x\le X^{1/22}}\#\{u\}\ \ll\ X^{1/22}\cdot X^{1/2}=X^{6/11}.
\]

If $x> X^{1/22}$, then $5x^{22}-4X\asymp x^{22}$ and the interval for $|u|$ has length
\[
\sqrt{5x^{22}+4X}-\sqrt{5x^{22}-4X}
=\frac{8X}{\sqrt{5x^{22}+4X}+\sqrt{5x^{22}-4X}}
\ll \frac{X}{x^{11}}.
\]
Thus, we have $\#\{u\}\ll X/x^{11}+1$.
In the range $X^{1/22}<x\le X^{1/11}$ one has $X/x^{11}\ge 1$, so $\#\{u\}\ll X/x^{11}$ there, and therefore
\[
\sum_{X^{1/22}<x\le X^{1/11}}\#\{u\}\ \ll\ \sum_{X^{1/22}<x\le X^{1/11}}\frac{X}{x^{11}}
\ll\ X\int_{X^{1/22}}^\infty t^{-11}\,dt
\ll\ X\cdot (X^{1/22})^{-10}
= X^{6/11}.
\]
For $x> X^{1/11}$ the $u$-interval has length $\ll X/x^{11}<1$,
so there are at most $O(1)$ many possibilities $u$,
and the contribution from this regime is precisely $E_4(X)$.
Hence
\[
N_4(X)\ll X^{6/11}+E_4(X),
\]
and since $A_4(X)\le N_4(X)$ this proves the lemma.
\end{proof}

\section{The proof of Theorem~\ref{thm:main}}\label{sec:proof}

To prove Theorem~\ref{thm:main},
it remains to prove Lemma~\ref{lem:reduction} and Lemma~\ref{lem:main_engine}.

\begin{proof}[Proof of Lemma~\ref{lem:reduction}]
Let $X\ge 3$ and apply Proposition~\ref{prop:X2k}.
Split the contribution into $k=1$, $k=2$, and $k\ge 3$.

For $k=1$, primes arising from $X_2$ are contained in the set counted by $A_2(X)$, hence by Lemma
\ref{lem:X2count} they contribute $O(X^{13/22}+E_2(X))$ primes $\ell\le X$.
For $k=2$, primes arising from $X_4$ are contained in the set counted by $A_4(X)$, hence by Lemma
\ref{lem:X4count} they contribute $O(X^{6/11}+E_4(X))$ primes $\ell\le X$.
For $k\ge 3$, Proposition \ref{prop:xiong} gives
\[
\#(\mathcal P\cap X_{2k}\cap[-X,X])\ \ll\ X^{1/2}
\qquad \left(3\le k<\frac{\log X}{2\log 2}\right),
\]
and the remaining $k$ contribute nothing. Summing over $k$ yields an additional factor $O(\log X)$.
Therefore, we have
\[
S(X)\ \ll\ X^{13/22}\ +\ X^{6/11}\ +\ X^{1/2}\log X\ +\ E_2(X)\ +\ E_4(X),
\]
as claimed.
\end{proof}

\begin{proof}[Proof of Lemma~\ref{lem:main_engine}]
Fix $\eta>0$ and assume the $abc$ Conjecture.

\smallskip\noindent
\emph{Proof of (i).}
Let $(x,y)\in\mathbb Z_{\ge 1}\times\mathbb Z$ contribute to $E_2(X)$, and put $k:=y^2-x^{11}$.
Then $1\le |k|\le X$ and $x>X^{2/11}$. Set $a:=x^{11}$, $b:=k$, $c:=y^2$, so that $a+b=c$.
Let $d=\gcd(a,b)$ and write $a=da_1$, $b=db_1$, $c=dc_1$ with $\gcd(a_1,b_1,c_1)=1$.
Applying the $abc$ Conjecture to $a_1+b_1=c_1,$ we obtain
\[
\frac{y^2}{d}=|c_1|\ \ll_\varepsilon\ \rad(a_1b_1c_1)^{1+\varepsilon}
\ \le\ \rad(abc)^{1+\varepsilon}
=\rad(x^{11}ky^2)^{1+\varepsilon}
\ \le\ |xky|^{1+\varepsilon}.
\]
Since $d\mid k$ one has $1\le d\le |k|\le X$, and therefore
\begin{equation}\label{eq:E2-abcbound}
y^2\ \ll_\varepsilon\ X^{2+\varepsilon}\, x^{1+\varepsilon}\, |y|^{1+\varepsilon}.
\end{equation}
Since $x>X^{2/11}$, we have
$$
y^2=x^{11}+k\geq x^{11}-X>0.
$$
As $y\neq 0$, dividing \eqref{eq:E2-abcbound} by $|y|^{1+\varepsilon}$ yields
\[
|y|^{1-\varepsilon}\ \ll_\varepsilon\ X^{2+\varepsilon}\, x^{1+\varepsilon}.
\]
Since $x>X^{2/11}$ implies $x^{11}>X^2\ge 2X$ for $X\ge 2$,
we have $ y^2 = x^{11} + k \ge $ $x^{11}-X\ge \tfrac12 x^{11}$ and hence
\[
|y|\ge \frac{1}{\sqrt2}\,x^{11/2}.
\]
Combining these inequalities gives
\[
x^{\frac{11}{2}(1-\varepsilon)-(1+\varepsilon)}\ \ll_\varepsilon\ X^{2+\varepsilon}.
\]
The exponent on the left equals $\frac92-\frac{13}{2}\varepsilon$, which is positive for $\varepsilon<9/13$.
Choosing $\varepsilon=\varepsilon(\eta)>0$ sufficiently small, this implies
\[
x\ \ll_\eta\ X^{4/9+\eta}.
\]
For each such $x$ there are at most $O(1)$ many possibilities for $y$, so
\[
E_2(X)\ \ll_\eta\ X^{4/9+\eta}.
\]

\smallskip\noindent
\emph{Proof of (ii).}
Let $(x,u)\in\mathbb Z_{\ge 1}\times\mathbb Z$ contribute to $E_4(X)$, and put $k:=u^2-5x^{22}$.
Then $1\le |k|\le 4X$ and $x>X^{1/11}$. Set $a:=5x^{22}$, $b:=k$, $c:=u^2$, so that $a+b=c$.
As before, let $d=\gcd(a,b)$ and write $a=da_1$, $b=db_1$, $c=dc_1$ with $\gcd(a_1,b_1,c_1)=1$.
Applying the $abc$ Conjecture to $a_1+b_1=c_1$, we obtain
\[
\frac{u^2}{d}=|c_1|\ \ll_\varepsilon\ \rad(a_1b_1c_1)^{1+\varepsilon}
\ \le\ \rad(abc)^{1+\varepsilon}
=\rad(5x^{22}ku^2)^{1+\varepsilon}
\ \ll\ |xku|^{1+\varepsilon}.
\]
Since $d\mid k$ and $|k|\le 4X$, we get
\begin{equation}\label{eq:E4-abcbound}
u^2\ \ll_\varepsilon\ X^{2+\varepsilon}\, x^{1+\varepsilon}\, |u|^{1+\varepsilon}.
\end{equation}
Dividing by $|u|^{1+\varepsilon}$ yields
\[
|u|^{1-\varepsilon}\ \ll_\varepsilon\ X^{2+\varepsilon}\, x^{1+\varepsilon}.
\]
This division is justified because for large enough $X$ (and $x > X^{1/11}$), we get
\[ u^2=5x^{22}+k \ge 5x^{22}-4X \ge 5x^{22} - 4x^{11} \ge \frac52 x^{22}
  \implies |u|\ge \sqrt{\frac52}\,x^{11}.  \]
Combining these inequalities, we obtain
\[
x^{11(1-\varepsilon)-(1+\varepsilon)}\ \ll_\varepsilon\ X^{2+\varepsilon}.
\]
The exponent on the left equals $10-12\varepsilon$, which is positive for $\varepsilon<5/6$.
Choosing $\varepsilon=\varepsilon(\eta)>0$ sufficiently small, this implies
\[
x\ \ll_\eta\ X^{1/5+\eta}.
\]
For each such $x,$ there are at most $O(1)$ many possibilities for $u$ (usually just $\pm |u|$), so
\[
E_4(X)\ \ll_\eta\ X^{1/5+\eta}. \qedhere
\]
\end{proof}

\begin{proof}[Proof of Theorem \ref{thm:main}]
Assuming the $abc$ Conjecture, Lemma~\ref{lem:reduction} and Lemma~\ref{lem:main_engine} imply, for any fixed
$\eta>0$, that
\[
S(X)\ \ll_\eta\ X^{13/22}\ +\ X^{4/9+\eta}\ +\ X^{1/5+\eta}.
\]
Now choosing any $\eta < \frac{13}{22} - \frac{4}{9} = \frac{29}{198}$ ends the proof.
\end{proof}

\bigskip

\section{The expected order of $S(X)$}\label{heuristic}

Here we offer a simple heuristic suggesting that $S(X)$ should be infinite, but extremely sparse.
As mentioned earlier (see Proposition~\ref{prop:X2k}),
if $\tau(n)$ is an odd prime then
\[ n=p^{2m} \]
for some odd prime $p$ and some integer $m\ge 1$. Next write
\[
\tau(p)=2p^{11/2}\cos\theta_p,
\]
where Deligne's bound gives $\left\lvert \cos\theta_p \right\rvert \le 1$. By the Hecke recurrence, we have
\[
\tau(p^r)=p^{11r/2}U_r(\cos\theta_p),
\]
where $U_r$ is the Chebyshev polynomial of the second kind (for example, see \cite{MasonHandscomb2002}). In particular, we have
\[
\tau(p^{2m})=p^{11m}U_{2m}(\cos\theta_p).
\]
The proof of the Sato--Tate Conjecture (see \cite{BarnetLambGeeGeraghty2011})
implies that the angles $\theta_p$ are equidistributed in $[0,\pi]$
with respect to the measure $\frac{2}{\pi}\sin^2\theta\,d\theta$.
Therefore, for each fixed $m\ge 1$, a positive proportion (depending on $m$) of primes $p$ satisfy
\[
|U_{2m}(\cos\theta_p)|\asymp 1,
\]
and hence
\[
|\tau(p^{2m})|\asymp p^{11m}.
\]

Fix $m\ge 1$. The condition $|\tau(p^{2m})|\le X$ then corresponds, for a positive proportion of primes $p$, to the range
\[
p\ll X^{\frac{1}{11m}}.
\]
By the Prime Number Theorem, the number of such primes is of order
\[
\pi\bigl(X^{\frac{1}{11m}}\bigr)\asymp \frac{X^{\frac{1}{11m}}}{\log X}.
\]
Treating the integers $|\tau(p^{2m})|$ of size about $X$ as having prime probability about $1/\log X$, one is led to the layer-by-layer estimate
\[
S_m(X):=\#\{\ell\le X:\ \ell\ \text{prime and } |\tau(p^{2m})|=\ell\ \text{ for some prime }p\}
\asymp \frac{X^{\frac{1}{11m}}}{(\log X)^2}.
\]
The dominant contribution comes from $m=1$, namely we have
\[
S_1(X)\asymp \frac{X^{1/11}}{(\log X)^2},
\]
while each higher layer is smaller. This leads to the heuristic prediction
\[
S(X) \asymp \frac{X^{\frac{1}{11}}}{(\log X)^2}
\]
up to local factors, and in particular suggests that there should be infinitely many prime values of $|\tau(n)|$.

\bigskip

\section{AxiomProver's autonomous Lean verification}\label{sec:AI}
We provide context for this project as well as the protocol used for Lean
formalization and verification. We gave AxiomProver the informal statements in the paper and asked whether AxiomProver
can generate autonomously prove and formalize the results assuming the $abc$ Conjecture and relevant results from existing literature, offering an example of AI assistance in mathematical research.
What did we learn? We found that AxiomProver could complete the task.
To be precise, AxiomProver autonomously proved and formalized
Lemma~\ref{lem:reduction}, Lemma~\ref{lem:main_engine},
and Theorem~\ref{thm:main}, all assuming Proposition 5.4 of \cite{Xiong2025}.

\subsection*{AxiomProver Protocol}
Here we describe the protocol we employed using AxiomProver to
autonomously verify Lemma~\ref{lem:reduction}, Lemma~\ref{lem:main_engine},
and Theorem~\ref{thm:main} in Lean with mathlib (see \cite{Lean, Mathlib2020}),
the main result in the paper.

\subsection*{Process}
The formal proofs provided in this work were developed and verified using Lean \textbf{4.26.0}.
Compatibility with earlier or later versions is not guaranteed due to the
evolving nature of the Lean 4 compiler and its core libraries.
The relevant files are all posted in the following repository:
\begin{center}
  \url{https://github.com/AxiomMath/ramanujan-tau-misses-primes}
\end{center}
The input files were
\begin{itemize}
  \item \texttt{informal\_statement.tex}, the problem statements of
   Theorem~\ref{thm:main}, Lemma~\ref{lem:reduction} and Lemma~\ref{lem:main_engine}
   in natural language
  \item a configuration file \texttt{.environment} that contains the single line
  \begin{quote}
    \slshape
    lean-4.26.0
  \end{quote}
  which specifies to AxiomProver which version of Lean should be used.
  \item a markdown file \texttt{task.md} describing the current task.
  \item three LaTeX source files of \cite{ BalakrishnanCraigOno2022, BCOT2023} and \cite{Xiong2025}.
  \item \texttt{requirement.md} that instructs AxiomProver to assume the $abc$ Conjecture and Proposition 5.4 of \cite{Xiong2025} whenever necessary.
 \end{itemize}
Given these files,
AxiomProver autonomously provided the following output files:
\begin{itemize}
  \item \texttt{problem.lean}, a Lean 4.26.0 formalization of the problem statement; and
  \item \texttt{solution.lean}, a complete Lean 4.26.0 formalization of the proof.
\end{itemize}
After AxiomProver generated a solution, the human authors wrote this paper
(without the use of AI) for human readers.
At first glance, the proofs found by AxiomProver do not resemble the narrative presented in this paper.
Turning a Lean file into a human-readable proof is difficult
because Lean is written as code for a type-checker.

\bibliographystyle{amsplain}

\end{document}